\documentclass{aptpub}

\pdfoutput=1

\usepackage{subfigure}
\usepackage{epsfig}
\usepackage{graphicx}
\usepackage[authoryear,longnamesfirst]{natbib}
\RequirePackage{hypernat}
\RequirePackage[colorlinks,citecolor=blue,urlcolor=blue]{hyperref}

\RequirePackage{hypernat}

\authornames{T. Appuhamillage and D. Sheldon} 
\shorttitle{Skew Brownian Motion} 

\newcommand{\ds}{\displaystyle} 
\newcommand{\ts}{\textstyle} 
\newcommand{\1}{\mathbf{1}}

\begin{document}

\title{First Passage Time of Skew Brownian Motion} 

\authorone[Oregon State University]{Thilanka Appuhamillage} 
\authortwo[Oregon State University]{Daniel Sheldon} 

\addressone{Department of Mathematics, Oregon State University, 
Corvallis, OR 97331 USA.\\
Email address: ireshara@math.oregonstate.edu} 
\addresstwo{School of Electrical Engineering and Computer Science,
Oregon State University, Corvallis, OR 97331 USA.\\
Email address: sheldon@eecs.oregonstate.edu\\
Research supported in part by the grant DBI-0905885 from the National Science Foundation.} 

\begin{abstract}
Nearly fifty years after the introduction of skew Brownian motion by
\cite{Ito_McKean63}, the first passage time distribution remains unknown. 
In this paper, we generalize results of \cite{Yor} and \cite{Csaki} to
derive formulae for the distribution of ranked excursion heights of
skew Brownian motion. We then derive the first passage time
distribution as a simple corollary.
\end{abstract}



\section{Introduction}
In this paper, we obtain formulae for the distributions of first
passage time and ranked excursion heights of skew Brownian motion.
Since \cite{Ito_McKean63} first introduced skew Brownian motion,
numerous papers in the mathematics literature have highlighted the
special structure of the process and connected it to diverse
applications; see, for example,
\cite{walsh1}, \cite{LeGall}, \cite{Barlow89},
\cite{Ouknine}, \cite{decamps2006}, \cite{Ramirez06}, \cite{Ramirez08}, and \cite{Ramirez10}.

Skew Brownian motion is broadly applicable to diffusion problems in
which the diffusion coefficient is discontinuous in space.
For example, \cite{app_WRR} analyzed mathematical properties of
skew Brownian motion as they relate to the convection-dispersion
movement of solute through porous media in the presence of a sharp
interface. Their work was motivated by experiments in hydrology
demonstrating a fundamental asymmetry in the breakthrough
curves of solute crossing a sharp interface in opposite directions
\citep{Berkowitz09}.
In this context, the first passage time distribution of skew Brownian
motion describes the breakthrough times of solute injected on one side
of an interface and observed on the other side \citep{app_AAP}.

Similar situations arise problems in physical and natural
sciences. For example, Brownian motion has been widely
used, and critiqued, as a model of animal movement \citep{turchin1996},
\citep{blackwell}.
However, there is an emerging consensus among ecologists that landscape heterogeneity
is a necessary ingredient of movement models \citep{dalziel}. In
particular, sharp interfaces delimiting two
different movement regimes occur at the boundaries of habitat
patches \citep{Wiens1989,butterfly,turchin1991}, or at other
environmental discontinuities such as those in ocean temperatures
\citep{pinaud} or the level of surface chlorophyll in oceans
\citep{polovina}. 
In an ecological context, \cite{McKenzie} advocate the use of first
passage time to model the time required for an organism to first reach
a specified location in the landscape (see also \cite{Fauchald2003}),
and give as an example the foraging behavior of a predator searching
for stationary prey items.

The main contribution of the present paper is to derive the first
passage time distribution of skew Brownian motion. We achieve this by
first characterizing the distribution of ranked excursion heights of
skew Brownian motion. This result generalizes formuale of \cite{Yor}
for the distribution of ranked excursion heights of the standard
Brownian bridge, and analogous formulae presented by \cite{Csaki} for
Brownian motion.
We then apply our formulae for the ranked excursion heights of skew
Brownian motion to derive our main result on first passage time. 

The paper is organized as follows. In Section \ref{results}, we state
the main results. In Section \ref{coupled}, we develop a coupled
construction for two different skew Brownian motion
processes with different skew parameters that leads to an important
relationship between distributions of ranked excursion heights of the
two processes, stated in Theorem \ref{relation}. In Section
\ref{proofs}, we prove the main results as corollaries to Theorem
\ref{relation}. In Section \ref{examples}, we present several examples 
that demonstrate the calculations that are possible using the first
passage time density, and also the assymetry inherent in the first
passage time of particles crossing an interface in opposite directions.

\section{Preliminaries and Main Results}
\label{results}

To set some notation and basic definitions, let $B = \{B_t : t
\geq 0\}$ be the 
\emph{standard Brownian motion} process on a probability space $(\Omega,\mathcal{F},P)$ and let 
$J_1, J_2, \ldots$ denote the excursion intervals of the reflected
process $\{|B_t| : t \geq 0\}$.
For $\alpha \in (0,1)$, let $\{A_m^{(\alpha)}: m=0, 1, \ldots\}$
be a sequence of i.i.d. $\pm 1$ Bernoulli random variables with
$P(A_m^{(\alpha)} = 1) = \alpha$.
Define the \emph{$\alpha$-skew Brownian motion} process $B^{(\alpha)}$ started at $0$ 
by
\begin{equation}
  \label{skew}
  B_t^{(\alpha)} = \sum_{m = 1}^{\infty} \1_{J_m}(t) A_m^{(\alpha)} |B_t|,
\end{equation}
where ${\mathbf 1}_S$ denotes the  indicator function of the set $S$. \\

Now let
\[
M^{(\alpha)}_1(t) \geq M^{(\alpha)}_2(t) \geq ... \geq 0
\]
be the ranked decreasing sequence of excursion heights $\sup_{s \in
J_m \cap [0,t]} B_s^{(\alpha)}$ ranging over all $m$ such that $J_m
\cap [0,t]$ is nonempty. Note that a negative excursion has
height zero, and that the height of the final excursion is
included in the ranked list even if that excursion is incomplete. Our
first main result gives the distribution of ranked excursion heights.

\begin{thm}
\label{main-theorem}
Fix $y\geq 0$ and $t>0$. Then for each $j=1, 2, ...$, the distribution of $M_j^{(\alpha)}(t)$ is given by the formula
\begin{equation*}\label{dist1}
 P_0(M_j^{(\alpha)}(t) > y) = 
\sum_{h=1}^{\infty}
2\binom{h-1}{j-1}
\big(1 - \ts 2\alpha \big)^{h-j} 
\big(\ts 2\alpha \big)^j
\big(1-\Phi((2h-1)y/\sqrt{t})\big),
\end{equation*} 
where $\Phi(\cdot)$ is the standard normal distribution function.
\end{thm}


Now let
\[
T_y^{(\alpha)}= \text{inf}\{s\geq 0 : B_s^{(\alpha)}=y\}
\]
denote the first time for $\alpha$-skew Brownian motion to reach $y$
and let $f^{(\alpha)}(x,y,t)$ denote the first passage time density to $y$
at time $t$ of $\alpha$-skew Brownian motion started at $x$. When
$\alpha=1/2$, this is the well known first passage time density
$f(x,y,t)$ for Brownian motion (e.g., see page 30 of \cite{bhattway1}):
\begin{equation*}
  f^{(1/2)}(x,y,t)\equiv f(x,y,t)=\frac{|y-x|}{\sqrt{2\pi}~t^{3/2}}\exp\Bigl\{-~\frac{(y-x)^2}{2t}\Bigr\}.
\end{equation*} 

Notice that $f^{(\alpha)}(x,y,t)dt=P_x(T_y^{(\alpha)}\in dt)$. Our
second main result gives formulae for the first passage time density.

\begin{thm}
\label{first-passage-time}
Fix $t>0$. Then\\
\[
f^{(\alpha)}(x,y,t) = 
\begin{cases}
  g^{(\alpha)}_{x,y}(t) & \text{for}~~ x\leq 0<y\\
 \ds g^{(\alpha)}_{x,y}(t)+h_{x,y}(t)-\sum_{n=1}^{\infty}\frac{2}{\pi n}\sin\Bigl\{\frac{\pi (y-x) n}{y}\Bigr\}g_{0,y}^{(\alpha)}\ast \kappa_{n,y}(t)
       & \text{for}~~ 0<x<y\\
 f(x,y,t) & \text{for}~~ 0<y<x
\end{cases}
\]

and

\[
f^{(\alpha)}(x,y,t) = 
\begin{cases}
  g^{(1-\alpha)}_{-x,-y}(t) & \text{for}~~ y<0\leq x\\
 \ds g^{(1-\alpha)}_{-x,-y}(t)+h_{-x,-y}(t)-\sum_{n=1}^{\infty}\frac{2}{\pi n}\sin\Bigl\{\frac{\pi (y-x) n}{y}\Bigr\}g_{0,-y}^{(1-\alpha)}\ast \kappa_{n,-y}(t)
       & \text{for}~~ y<x<0\\
 f(-x,-y,t) & \text{for}~~ x<y<0
\end{cases}
\]

\vspace{0.3cm}
\noindent 
where $g_{x,y}^{(\alpha)}=\ds
2\alpha\sum_{j=1}^{\infty}(1-2\alpha)^{j-1}\frac{|x-(2j-1)y|}{\sqrt{2\pi}~t^{3/2}}\exp\Bigl\{-~\frac{(x-(2j-1)y)^2}{2t}\Bigr\}$
for $x<y$, 
the quantity $\kappa_{n,y}(t)$
is the density at time $t$ of the exponential distribution with parameter
$\lambda(n,y)=\frac{\pi^2n^2}{2y^2}$, %
and 
$\ds h_{x,y}(t)=\frac{\pi}{y^2}\sum_{n=1}^{\infty}n\exp\Bigl\{-~\frac{\pi^2n^2t}{2y^2}\Bigr\}\sin\Bigl\{\frac{\pi(y-x)n}{y}\Bigr\}$
is the well known formula (e.g. see page 296 in \cite{feller}) for the probability that Brownian motion started at
$x$ reaches zero before reaching $y$, and that this event occurs in the time
interval $dt$.
\end{thm}

The cases $0<y<x$ and $x<y<0$ in the first passage time density are clear because 
all paths starting at $x$ reach $y$ without hitting zero and
hence they are all Brownian motion paths. 
In cases, $x \leq 0 < y$ and $y<0\leq x$, all paths must cross zero and
densities are obtained as a straightforward corollary of Theorem
\ref{main-theorem} and the fact that $P_0(T^{(\alpha)}_y\in dt)=P_0(T^{(1-\alpha)}_{-y}\in dt)$. 
The situation is most complicated when $0 < x < y$ and $y<x<0$, 
and one must consider two types of paths from $x$ to $y$: those
that cross zero, and those that reach $y$ before they reach zero.
 
Notice that when $\alpha=1/2$, we recover existing results for
standard Brownian motion. Namely, from Theorem \ref{main-theorem}, we recover
the distribution of ranked excursion heights stated in Theorem 3.1 of 
\cite{Csaki}, and from Theorem \ref{first-passage-time} we
recover the well known first passage time distribution of standard
Brownian motion (this fact is not immediately obvious, but nonetheless
true, in most complicated cases when $0 < x < y$ and $y<x<0$).

\section{Relating excursion heights for $B^{(\alpha)}$ and $B^{(\beta)}$}
\label{coupled}

Let $0\leq\alpha<\beta\leq 1$. Consider the following coupled
construction of $\alpha$-skew and $\beta$-skew Brownian motion.
Let $B$ be the standard Brownian motion process and
let $A^{(\beta)}=\{A_m^{(\beta)}:m=0,1,\ldots\}$ be independently chosen excursion
signs so that Equation \eqref{skew} yields an instance of $\beta$-skew
Brownian motion.

Next,  let $\{A_m^{(\alpha/\beta)}$: $m=0, 1, \ldots\}$ be a sequence
of i.i.d. $\pm 1$ Bernoulli random variables independent of $A^{(\beta)}$ and
$B$ with $P(A_m^{(\alpha/\beta)} = 1) = \alpha/\beta$.
Define $A_m^{(\alpha)}$ as follows
\[
A_m^{(\alpha)} = 
\begin{cases}
  1& A_m^{(\beta)} = 1, A_m^{(\alpha/\beta)} = 1, \\
 -1& \text{otherwise}.
\end{cases}
\]
By construction, the sequence $\{A_m^{(\alpha)} : m = 0, 1,
\ldots\}$ consists of i.i.d. $\pm 1$ Bernoulli random variables that are
independent of $B$ with $P(A_m^{(\alpha)} = 1) = \alpha$. Hence, by
using the variables $A_m^{(\alpha)}$ as the excursion signs in
Equation \eqref{skew}, we obtain an instance $B^{(\alpha)}$ of
$\alpha$-skew Brownian motion.

We think of this as a two-step process: first, construct $B^{(\beta)}$
by independently setting each excursion of $|B|$ to be positive with
probability $\beta$; then, for each positive excursion
of $B^{(\beta)}$, independently decide whether to keep it positive
(with probability $\alpha/\beta$), or flip it to be negative (with
probability $1 - \alpha/\beta$). 

The following theorem is motivated by this coupled construction.


\begin{thm}
\label{relation}
Fix $y\geq 0$, $t>0$ and $\alpha,\beta \in (0,1)$. For each $j=1, 2,
...$, the following relation between ranked excursion heights of
$\alpha$- and $\beta$-skew Brownian motions holds.

\begin{equation}
  \label{excursionRelation}
P_0(M_j^{(\alpha)}(t) > y) = 
\sum_{h=1}^{\infty} 
\binom{h-1}{j-1} \Big(1 - \ts\frac{\alpha}{\beta}\Big)^{h-j} 
\Big(\ts \frac{\alpha}{\beta} \Big)^j
P_0(M_h^{(\beta)}(t) > y).
\end{equation}

\end{thm}

Before giving the proof of Theorem \ref{relation}, we state the following lemma from \cite{Yor}, as we use it in the proof. 

\begin{lem}[\cite{Yor}, Lemma 9]
\label{invert}
Let 
\[
b_k = \sum_{m=0}^{\infty} \binom{m}{k} a_m, \quad k=0, 1, ...
\]
be the binomial moments of a nonnegative sequence $(a_m,~m=0, 1, ...)$. Let $B(\theta):=\sum_{k=0}^{\infty}b_k\theta^k$ and suppose $B(\theta_1)<\infty$ for some $\theta_1>1$. Then
\[
a_m=\sum_{k=0}^{\infty} (-1)^{k-m}\binom{k}{m} b_k, \quad m=0, 1, ...,
\]
where the series is absolutely convergent. 
\end{lem}

\begin{proof}[Proof of Theorem \ref{relation}]

For $\alpha < \beta$, we have by the two-step construction of the
excursion sign $A_m^{(\alpha)}$ that 
$M_j^{(\alpha)}(t) = M_{H_j}^{(\beta)}(t)$, 
where $H_j$ has a negative binomial distribution:
\[
P(H_j = h) = 
\binom{h-1}{j-1} \Big(1 - \ts\frac{\alpha}{\beta}\Big)^{h-j} 
\Big( \frac{\alpha}{\beta} \Big)^j.
\]
Hence
\begin{equation}
P_0(M_j^{(\alpha)}(t) > y) = 
\sum_{h=1}^{\infty} 
\binom{h-1}{j-1} \Big(1 - \ts\frac{\alpha}{\beta}\Big)^{h-j} 
\Big(\ts \frac{\alpha}{\beta} \Big)^j
P_0(M_h^{(\beta)}(t) > y).
\end{equation}
\\
\noindent
For $\beta < \alpha$, the relation can be inverted by an application
of Lemma \ref{invert}. Let $k:=j-1$ and $m:=h-1$. Then one can write
\eqref{excursionRelation} as

\begin{equation}
P_0(M_{k+1}^{(\alpha)}(t) > y) = 
\sum_{m=0}^{\infty} 
\binom{m}{k} \Big(1 - \ts\frac{\alpha}{\beta}\Big)^{m-k} 
\Big(\ts \frac{\alpha}{\beta} \Big)^{k+1}
P_0(M_{m+1}^{(\beta)}(t) > y).
\end{equation}
We then apply Lemma \ref{invert} to the sequences
\[
b_k := \Big(1 - \ts\frac{\alpha}{\beta}\Big)^k\Big(\ts \frac{\alpha}{\beta} \Big)^{-k-1}P_0(M_{k+1}^{(\alpha)}(t) > y),
\quad a_m := \Big(1 - \ts\frac{\alpha}{\beta}\Big)^m P_0(M_{m+1}^{(\beta)}(t) > y).
\] 
After simplifying, we obtain
\begin{equation}
P_0(M_j^{(\beta)}(t) > y) = 
\sum_{h=1}^{\infty}
\binom{h-1}{j-1} \Big(1 - \ts\frac{\beta}{\alpha}\Big)^{h-j} 
\Big(\ts \frac{\beta}{\alpha} \Big)^j
P_0(M_h^{(\alpha)}(t) > y).
\end{equation}

\end{proof}

\section{Proofs of Main Theorems}
\label{proofs}

We now observe that the main results announced in the introduction will follow as corollaries to Theorem \ref{relation}. We first prove the Theorem \ref{main-theorem}.

\begin{proof}[Proof of Theorem \ref{main-theorem}]
By Theorem 3.1 in 
\cite{Csaki}, 
\begin{equation}
P_0(M_j^{(1/2)}(t) > y) = 2\Big(1 - \Phi\Big((2j - 1)\frac{y}{\sqrt{t}}\Big)\Big)
\end{equation}
The result is immediate from Theorem \ref{relation} by taking $\beta=1/2$ in \eqref{excursionRelation}.
\end{proof}
 
We now use Theorem \ref{main-theorem} and the following corollary to compute the distribution of first
passage time asserted in Theorem \ref{first-passage-time}.

\begin{cor}
\label{corollary-first-passage-time-from-0}
Fix $t>0$. Then
\begin{equation}
\label{first-passage-time-from-0}
P_0(T_y^{(\alpha)} \in dt) = 
\begin{cases}
 2\alpha\sum_{h=1}^{\infty}(1-2\alpha)^{h-1}\frac{(2h-1)y}{\sqrt{2\pi}~t^{3/2}}\exp\Bigl\{-~\frac{((2h-1)y)^2}{2t}\Bigr\}~dt & \text{for} ~y>0\\
 2(1-\alpha)\sum_{h=1}^{\infty}(2\alpha-1)^{h-1}\frac{(2h-1)(-y)}{\sqrt{2\pi}~t^{3/2}}\exp\Bigl\{-~\frac{((2h-1)y)^2}{2t}\Bigr\}~dt & \text{for} ~y<0
\end{cases}
\end{equation}
\end{cor}

\begin{proof}
For the case $y>0$ and $t>0$, we have the following relation between the distributions of 
$T_y^{(\alpha)}$ and the highest excursion of skew Brownian motion started at $0$:
\[
P_0(T_y^{(\alpha)} < t) = P_0(M_1^{(\alpha)}(t) > y).
\]

Thus using Theorem \ref{main-theorem}, one has
\begin{equation*}
  \begin{split}
     P_0(T_y^{(\alpha)} < t) 
     &= P_0(M_1^{(\alpha)}(t) > y)\\
     &= \displaystyle 4\alpha\sum_{h=1}^{\infty} (1-2\alpha)^{h-1}
     \int_{\frac{(2h-1)y}{\sqrt{t}}}^{\infty} \frac{1}{\sqrt{2\pi}}\exp\Bigl\{-\frac{z^2}{2}\Bigr\}~dz.
  \end{split}
\end{equation*}
The result is immediate after taking the derivative of the above expression with respect to $t$.

For the case $y<0$, use the relation $P_0(T^{(\alpha)}_y\in dt)=P_0(T^{(1-\alpha)}_{-y}\in dt)$ and the case $y>0$.

\end{proof}

\begin{proof}[Proof of Theorem \ref{first-passage-time}]
Let $T_y\equiv T_y^{(1/2)}$ denote the first time for standard Brownian motion to reach $y$. 
Recalling that $\ds P_0(T_y\in dt)=\frac{|y|}{\sqrt{2\pi}~t^{3/2}}\exp\Bigl\{-~\frac{y^2}{2t}\Bigr\}~dt$; 
e.g. see page 30 in \cite{bhattway1}, one can write equation (\ref{first-passage-time-from-0}) as
\begin{equation}
\label{first-passage-time-from-0-simple}
P_0(T_y^{(\alpha)} \in dt) = 
\begin{cases}
 2\alpha\sum_{h=1}^{\infty}(1-2\alpha)^{h-1}P_0(T_{(2h-1)y}\in dt) & \text{for}~y>0\\
 2(1-\alpha)\sum_{h=1}^{\infty}(2\alpha-1)^{h-1}P_0(T_{(2h-1)y}\in dt) & \text{for}~y<0
\end{cases}
\end{equation}
 
Now note that $T_0^{(\alpha)}$ is distributed as $T_0$ under $P_x$ for $x\neq 0$, $0<\alpha < 1$. 
So clearly for $t>0$, one has
\[
P_x(T_0^{(\alpha)} > t) = P_x(T_0 > t).  
\]

We prove the case $y>0$ in the Theorem \ref{first-passage-time}. The case $y<0$ is similar.

\noindent {\it Case $x\leq 0<y$} :

Using the strong Markov property of skew Brownian motion,
\begin{equation}
\label{convolution}
\ds P_x(T_y^{(\alpha)} > t) =\int_0^t P_x(T_0 > t-s)P_0(T_y^{(\alpha)} \in ds).
\end{equation} 

Then from the first case of equation (\ref{first-passage-time-from-0-simple}), one has 
\begin{equation}\label{for-x-negative}
 \begin{split}
   \ds P_x(T_y^{(\alpha)} > t)
     &= 2\alpha\sum_{h=1}^{\infty}(1-2\alpha)^{h-1} \int_0^t P_x(T_0 > t-s)P_0(T_{(2h-1)y}\in ds) \\
     &= 2\alpha\sum_{h=1}^{\infty}(1-2\alpha)^{h-1} P_x(T_{(2h-1)y}>t) 
 \end{split}
\end{equation}  

By differentiating the above expression with respect to $t$ and recalling \\
$\ds P_x(T_y\in dt)= \frac{|y-x|}{\sqrt{2\pi}~t^{3/2}}\exp\Bigl\{-~\frac{(y-x)^2}{2t}\Bigr\}~dt$, one has
\begin{equation}
 \ds P_x(T_y^{(\alpha)}\in dt)=
 2\alpha\sum_{h=1}^{\infty}(1-2\alpha)^{h-1}\frac{|x-(2h-1)y|}{\sqrt{2\pi}~t^{3/2}}\exp\Bigl\{-~\frac{(x-(2h-1)y)^2}{2t}\Bigr\}~dt
\end{equation}  

\noindent {\it Case $0<x<y$} :

Observe that
\begin{equation}
\label{1}
P_x(T_y^{(\alpha)}\in dt)=P_x(T_y^{(\alpha)}\in dt , (T_0^{(\alpha)}\leq t))
+ P_x(T_y^{(\alpha)}\in dt , (T_0^{(\alpha)}> t))
\end{equation}

We state the following formula for $P_x(T_0\in dt , (T_y > t))$; e.g. see page 296 in \cite{feller}, as we use it to 
compute (\ref{1})  
\begin{equation}\label{2}
\ds P_x(T_0\in dt , (T_y > t)) = 
\frac{\pi}{y^2}\sum_{n=1}^{\infty}n\exp\Bigl\{-~\frac{\pi^2n^2t}{2y^2}\Bigr\}\sin\Bigl\{\frac{\pi x n}{y}\Bigr\}~dt
\end{equation} 

Since the skew Brownian motion is Brownian motion until it reaches zero for the first time and from the reflection 
principle of Brownian motion, one can write the second term of the right hand side of equation (\ref{1}) using equation (\ref{2}) as follows:
\begin{equation}\label{3}
 \begin{split}
   P_x(T_y^{(\alpha)}\in dt , (T_0^{(\alpha)}> t))
    &= P_x(T_y\in dt , (T_0> t))\\
    &= P_{y-x}(T_0\in dt , (T_y>t))\\
    &= \ds \frac{\pi}{y^2}\sum_{n=1}^{\infty}n\exp\Bigl\{-~\frac{\pi^2n^2t}{2y^2}\Bigr\}\sin\Bigl\{\frac{\pi (y-x) n}{y}\Bigr\}~dt
 \end{split}
\end{equation}

For the first term of the right hand side of equation (\ref{1}), notice that
\begin{multline*}
 P_x(T_0^{(\alpha)}<T_y^{(\alpha)}<t) \\
  \begin{split}
   &= \ds\mathbb{E}_x\Bigl[\textbf{1}_{[T_0^{(\alpha)}<T_y^{(\alpha)}<t]}\Bigr]\\
   &= \ds\mathbb{E}_x\left[\mathbb{E}_x\Bigl[\textbf{1}_{[T_0^{(\alpha)}<T_y^{(\alpha)}<t]}
         \vert T_0^{(\alpha)}, \textbf{1}_{[T_0^{(\alpha)}<T_y^{(\alpha)}]}\Bigr]\right]\\
   &= \ds\int_0^t\mathbb{E}_x\Bigl[\textbf{1}_{[T_0^{(\alpha)}<T_y^{(\alpha)}<t]}\vert T_0^{(\alpha)}=s, 
          \textbf{1}_{[T_0^{(\alpha)}<T_y^{(\alpha)}]}\Bigr]
        P_x(T_0^{(\alpha)}\in ds, (T_0^{(\alpha)}<T_y^{(\alpha)}))\\
   &~~~~+~\ds\int_0^t\mathbb{E}_x\Bigl[\textbf{1}_{[T_0^{(\alpha)}<T_y^{(\alpha)}<t]}\vert T_0^{(\alpha)}=s, 
            \textbf{1}_{[T_0^{(\alpha)}\geq T_y^{(\alpha)}]}\Bigr]
        P_x(T_0^{(\alpha)}\in ds, (T_0^{(\alpha)}\geq T_y^{(\alpha)}))
  \end{split}
\end{multline*}

Using the strong Markov property of skew Brownian motion and since 
$\mathbb{E}_x\Bigl[\textbf{1}_{[T_0^{(\alpha)}<T_y^{(\alpha)}<t]}\vert T_0^{(\alpha)}=s, 
            \textbf{1}_{[T_0^{(\alpha)}\geq T_y^{(\alpha)}]}\Bigr]=0$, one has
\begin{equation*}
 \begin{split}
   P_x(T_0^{(\alpha)}<T_y^{(\alpha)}<t)
   & = \ds\int_0^t\mathbb{E}_0\Bigl[\textbf{1}_{[T_y^{(\alpha)}<t-s]}\Bigr]
        P_x(T_0^{(\alpha)}\in ds, (T_0^{(\alpha)}<T_y^{(\alpha)}))\\
   &= \ds\int_0^tP_0(T_y^{(\alpha)}<t-s)P_x(T_0^{(\alpha)}\in ds, (T_0^{(\alpha)}<T_y^{(\alpha)}))
 \end{split}
\end{equation*} 

Using equation (\ref{first-passage-time-from-0-simple}) and again from the fact that 
the skew Brownian motion is Brownian motion until it reaches zero for the first time, and 
$P_x(T_0^{\alpha}<T_y^{\alpha})=P_x(T_0<T_y)$ (note here $0<x<y$), one has
\begin{multline}
 \begin{split}
  P_x(T_0^{(\alpha)}<T_y^{(\alpha)}<t)
   &= \int_0^t 2\alpha\sum_{h=1}^{\infty}(1-2\alpha)^{h-1}P_0(T_{(2h-1)y}<t-s)
        P_x(T_0 \in ds , (T_0<T_y)) \\
   &= 2\alpha\sum_{h=1}^{\infty}(1-2\alpha)^{h-1}~\int_0^t P_0(T_{(2h-1)y}<t-s) 
          P_x(T_0 \in ds , (T_0<T_y)) \\
   &= 2\alpha\sum_{h=1}^{\infty}(1-2\alpha)^{h-1}~\int_0^t P_0(T_{(2h-1)y}<t-s)
         P_x(T_0 \in ds) \\
   &~~~-~ 2\alpha\sum_{h=1}^{\infty}(1-2\alpha)^{h-1}~\int_0^t P_0(T_{(2h-1)y}<t-s) 
          P_x(T_0 \in ds , (T_0>T_y)) \\
   &= 2\alpha\sum_{h=1}^{\infty}(1-2\alpha)^{h-1}P_x(T_{(2h-1)y} < t)\\
   &~~~-~2\alpha\sum_{h=1}^{\infty}(1-2\alpha)^{h-1}~\int_0^t P_0(T_{(2h-1)y}<t-s) 
          P_x(T_0 \in ds , (T_0>T_y)) 
 \end{split}
\end{multline}

For $h\geq 1$, the convolution integral in the sum of the second term in the above equation 
can be written using (\ref{2}) as;

\begin{multline}
 \int_0^t P_0(T_{(2h-1)y}<t-s)P_x(T_0 \in ds , (T_0>T_y))\\
 \begin{split}
  &= \int_0^t P_x(T_0 <t-s , (T_0>T_y))P_0(T_{(2h-1)y}\in ds)\\
  &= \int_0^t \sum_{n=1}^{\infty}\frac{2}{\pi n}\sin\Bigl\{\frac{\pi (y-x) n}{y}\Bigr\}P_0(T_{(2h-1)y}\in ds)\\
  &~~~-~\int_0^t \sum_{n=1}^{\infty}\frac{2}{\pi n}
          \exp\Bigl\{-~\frac{\pi^2n^2(t-s)}{2y^2}\Bigr\}\sin\Bigl\{\frac{\pi (y-x) n}{y}\Bigr\}P_0(T_{(2h-1)y}\in ds)\\
  &= \sum_{n=1}^{\infty}\frac{2}{\pi n}\sin\Bigl\{\frac{\pi (y-x) n}{y}\Bigr\}P_0(T_{(2h-1)y}<t)\\
  &~~~-~\sum_{n=1}^{\infty}\frac{2}{\pi n}\exp\Bigl\{-~\frac{\pi^2n^2t}{2y^2}\Bigr\}\sin\Bigl\{\frac{\pi (y-x) n}{y}\Bigr\}
           \int_0^t \exp\Bigl\{\frac{\pi^2n^2s}{2y^2}\Bigr\}P_0(T_{(2h-1)y}\in ds)
 \end{split}
\end{multline}

Then one has

\begin{multline}
 P_x(T_0^{(\alpha)}<T_y^{(\alpha)}<t) \\
 \begin{split}
   &= 2\alpha\sum_{h=1}^{\infty}(1-2\alpha)^{h-1}P_x(T_{(2h-1)y} < t)\\
   &~~~-~2\alpha\sum_{h=1}^{\infty}\sum_{n=1}^{\infty}(1-2\alpha)^{h-1}\frac{2}{\pi n}
      \sin\Bigl\{\frac{\pi (y-x) n}{y}\Bigr\}P_0(T_{(2h-1)y}<t)\\
   &~~~+~2\alpha\sum_{h=1}^{\infty}\sum_{n=1}^{\infty}(1-2\alpha)^{h-1}\frac{2}{\pi n}
         \exp\Bigl\{-~\frac{\pi^2n^2t}{2y^2}\Bigr\}\sin\Bigl\{\frac{\pi (y-x) n}{y}\Bigr\}\\
   &~~~~~~~~~~~~~~~~~~~~~~~~~~~\times \int_0^t \exp\Bigl\{\frac{\pi^2n^2s}{2y^2}\Bigr\}P_0(T_{(2h-1)y}\in ds)
 \end{split}
\end{multline}

By differentiating the above equation with respect to $t$, one has

\begin{multline}\label{4}
 P_x(T_y^{(\alpha)}\in dt , (T_0^{(\alpha)}\leq T_y^{(\alpha)})) \\
 \begin{split}
   &= 2\alpha\sum_{h=1}^{\infty}(1-2\alpha)^{h-1}P_x(T_{(2h-1)y} \in dt)\\
   &~~~-~2\alpha\sum_{h=1}^{\infty}\sum_{n=1}^{\infty}(1-2\alpha)^{h-1}\frac{2}{\pi n}
      \sin\Bigl\{\frac{\pi (y-x) n}{y}\Bigr\}P_0(T_{(2h-1)y}\in dt)\\
   &~~~+~2\alpha\sum_{h=1}^{\infty}\sum_{n=1}^{\infty}(1-2\alpha)^{h-1}\frac{2}{\pi n}
         \exp\Bigl\{-~\frac{\pi^2n^2t}{2y^2}\Bigr\}\sin\Bigl\{\frac{\pi (y-x) n}{y}\Bigr\}\\
   &~~~~~~~~~~~~~~~~~~~~~~~~~~~\times \exp\Bigl\{\frac{\pi^2n^2t}{2y^2}\Bigr\}P_0(T_{(2h-1)y}\in dt)\\
   &~~~-~2\alpha\sum_{h=1}^{\infty}\sum_{n=1}^{\infty}(1-2\alpha)^{h-1}\frac{\pi n}{y^2}
         \exp\Bigl\{-~\frac{\pi^2n^2t}{2y^2}\Bigr\}\sin\Bigl\{\frac{\pi (y-x) n}{y}\Bigr\}\\
   &~~~~~~~~~~~~~~~~~~~~~~~~~~~\times \int_0^t \exp\Bigl\{\frac{\pi^2n^2s}{2y^2}\Bigr\}P_0(T_{(2h-1)y}\in ds)~dt
 \end{split}
\end{multline}

Recalling that $\ds P_x(T_y\in dt)=\frac{|y-x|}{\sqrt{2\pi}~t^{3/2}}\exp\Bigl\{-~\frac{(y-x)^2}{2t}\Bigr\}~dt$, 
and by (\ref{1}), (\ref{3}), (\ref{4}), one has

\begin{multline}
 P_x(T_y^{(\alpha)}\in dt)\\
 \begin{split}
  &= \ds 2\alpha\sum_{h=1}^{\infty}(1-2\alpha)^{h-1}\frac{|x-(2h-1)y|}{\sqrt{2\pi}~t^{3/2}}\exp\Bigl\{-~\frac{(x-(2h-1)y)^2}{2t}\Bigr\}~dt\\
  &~~~-~2\alpha\sum_{h=1}^{\infty}\sum_{n=1}^{\infty}(1-2\alpha)^{h-1}\frac{\pi n}{y^2}
         \exp\Bigl\{-~\frac{\pi^2n^2t}{2y^2}\Bigr\}\sin\Bigl\{\frac{\pi (y-x) n}{y}\Bigr\}\\
   &~~~~~~~~~~~~~~~~~~~~~~\times \int_0^t \exp\Bigl\{\frac{\pi^2n^2s}{2y^2}\Bigr\}
         \frac{(2h-1)y}{\sqrt{2\pi}~s^{3/2}}\exp\Bigl\{-~\frac{((2h-1)y)^2}{2s}\Bigr\}~ds~dt\\
  &~~~+~\ds \frac{\pi}{y^2}\sum_{n=1}^{\infty}n\exp\Bigl\{-~\frac{\pi^2n^2t}{2y^2}\Bigr\}\sin\Bigl\{\frac{\pi (y-x) n}{y}\Bigr\}~dt \\
  &=\ds 2\alpha\sum_{h=1}^{\infty}(1-2\alpha)^{h-1}\frac{|x-(2h-1)y|}{\sqrt{2\pi}~t^{3/2}}\exp\Bigl\{-~\frac{(x-(2h-1)y)^2}{2t}\Bigr\}~dt\\
  &~~~-\sum_{n=1}^{\infty}\frac{2}{\pi n}\sin\Bigl\{\frac{\pi (y-x) n}{y}\Bigr\}\\
  &~~~~~~~~\times \int_0^t\frac{\pi^2 n^2}{2y^2}\exp\Bigl\{-\frac{\pi^2n^2(t-s)}{2y^2}\Bigr\}
     ~2\alpha\sum_{h=1}^{\infty}(1-2\alpha)^{h-1}\frac{(2h-1)y}{\sqrt{2\pi}~s^{3/2}}
         \exp\Bigl\{-~\frac{((2h-1)y)^2}{2t}\Bigr\}~ds~dt\\
  &~~~+~\ds \frac{\pi}{y^2}\sum_{n=1}^{\infty}n\exp\Bigl\{-~\frac{\pi^2n^2t}{2y^2}\Bigr\}\sin\Bigl\{\frac{\pi (y-x) n}{y}\Bigr\}~dt 
 \end{split}
\end{multline}

\noindent {\it Case $0<y<x$} :

Notice that, in this case all skew Brownian motion paths till the first passage time to $y$ are away from $0$. 
Thus we have
\[
P_x(T_y^{(\alpha)}\in dt) =
  \frac{|y-x|}{\sqrt{2\pi}t^{3/2}}\exp\Bigl\{-~\frac{(y-x)^2}{2t}\Bigr\}~dt.
\]

\end{proof}

\section{Asymmetries in First Passage Time Density}
\label{examples}

A basic property of the first passage time density of regular Brownian
motion is that it is symmetric in $x$ and $y$, i.e., that $P_x(T_y\in
dt)=P_y(T_x\in dt)$ for any $x,y\in\mathbb{R}$ and $t\geq0$.
A practically important feature of skew Brownian motion is that it
introduces an asymmetry in the first passage time density for $x$ and
$y$ on opposite sides of the origin.
\cite{app_AAP} first demonstrated this by proving that there is a stochastic
ordering between the random variables with densities $P_{-y}(T_y^{(\alpha)} \in
dt)$ and $P_{y}(T_{-y}^{(\alpha)} \in dt)$ when $\alpha \neq 1/2$.
In simple terms, it takes longer to cross from negative to positive
than to cross from positive to negative when $\alpha < 1/2$, and the
opposite is true when $\alpha > 1/2$.

Our results quantify these relationships further by explicitly giving the
densities in each case. We conclude by illustrating the numerically
computed densities $P_{-1}(T_1^{(\alpha)} \in dt)$ and
$P_{1}(T_{-1}^{(\alpha)} \in dt)$ for different values of $\alpha$. 
In Figure \ref{ordering}(a) we see that
$P_{-1}(T_1^{(\alpha)} \in dt) < P_1(T_{-1}^{(\alpha)} \in dt)$ for
$\alpha < 1/2$. In Figure \ref{ordering}(b), we recover the symmetry of
Brownian motion when $\alpha=1/2$. In Figure \ref{ordering}(c), we see that
$P_{-1}(T_1^{(\alpha)} \in dt) > P_1(T_{-1}^{(\alpha)} \in dt)$ for
$\alpha > 1/2$.

\def \nudgeup {\vspace{-0.5cm}}
\def \nudgedown {\vspace{0.4cm}}

\begin{figure}[h]
\centering
  \begin{tabular}{c}
    \epsfig{figure=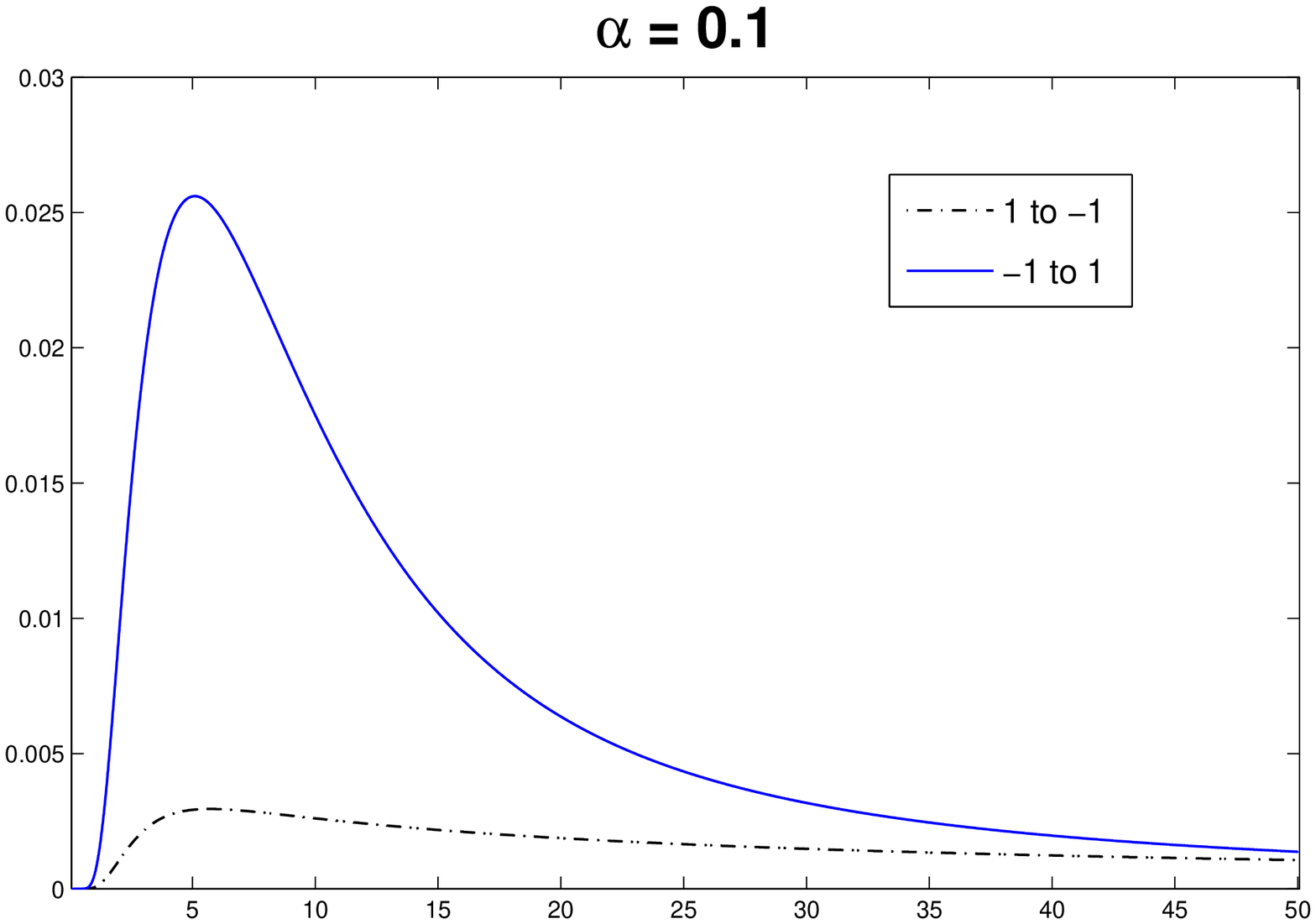,width=2in}\quad
    \epsfig{figure=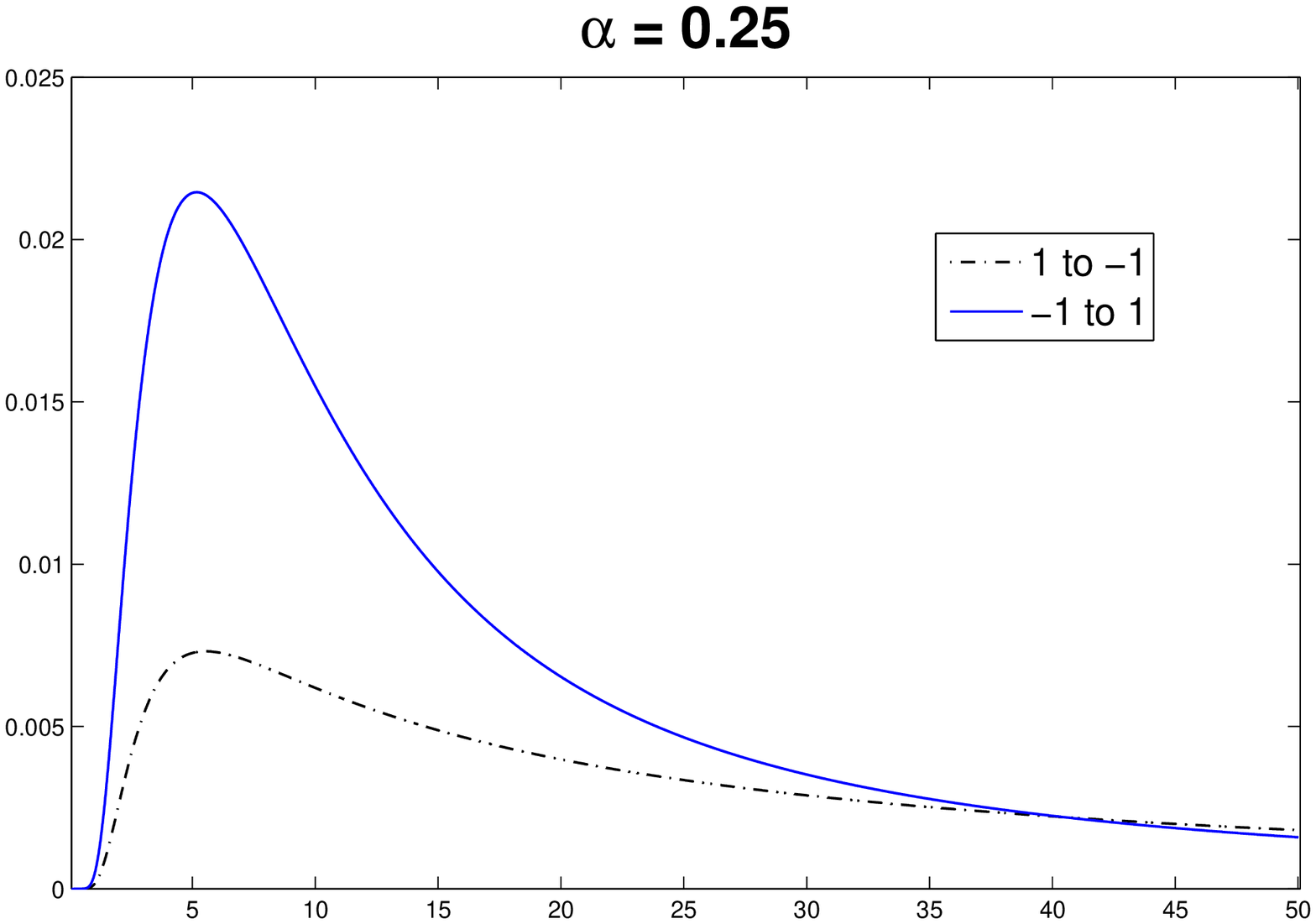,width=2in} \nudgeup \\
    (a) $\alpha < 1/2$ \nudgedown \\
    \includegraphics[width=2in]{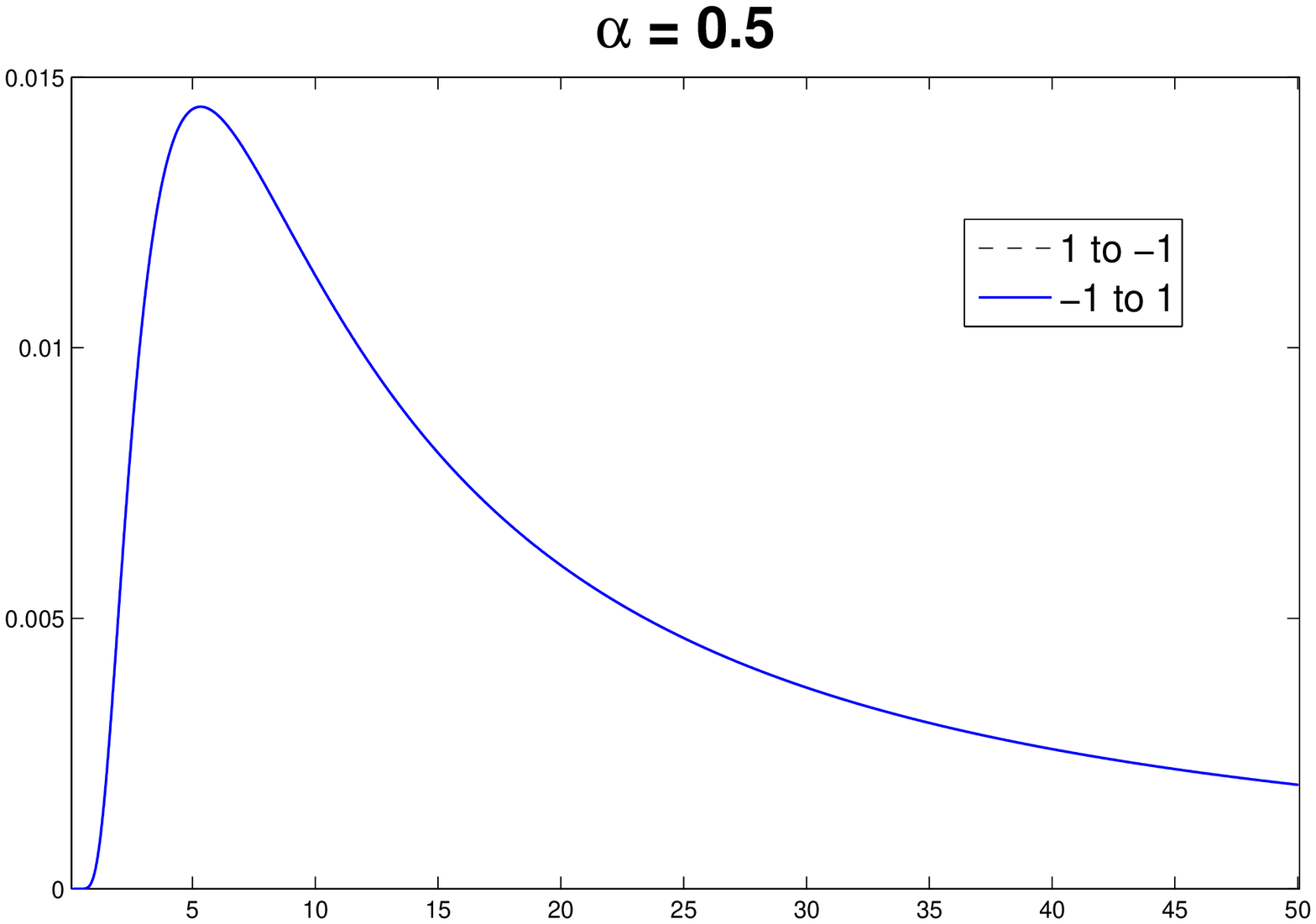} \nudgeup \\
    (b) $\alpha = 1/2$ \nudgedown \\
    \includegraphics[width=2in]{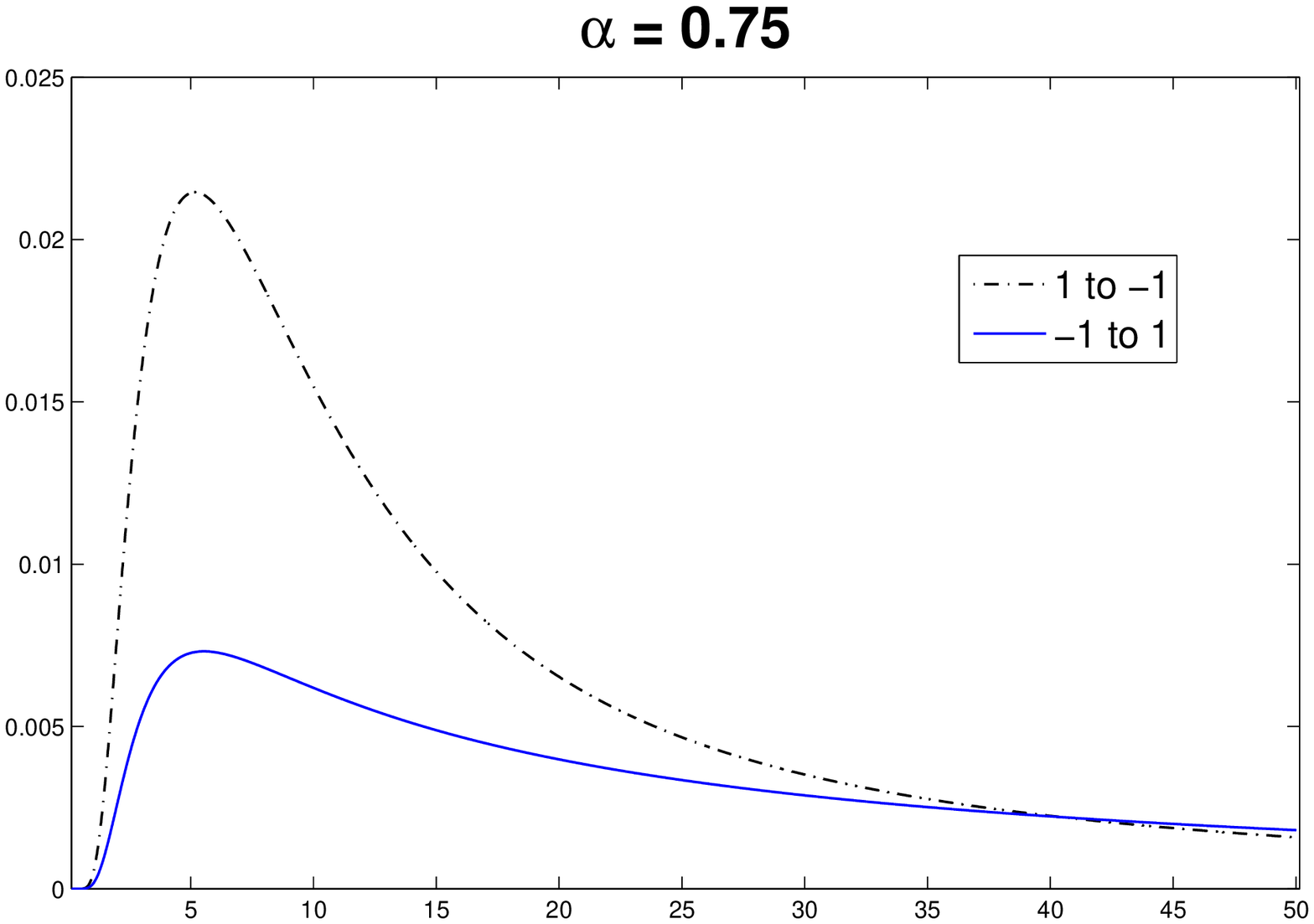}\quad
    \includegraphics[width=2in]{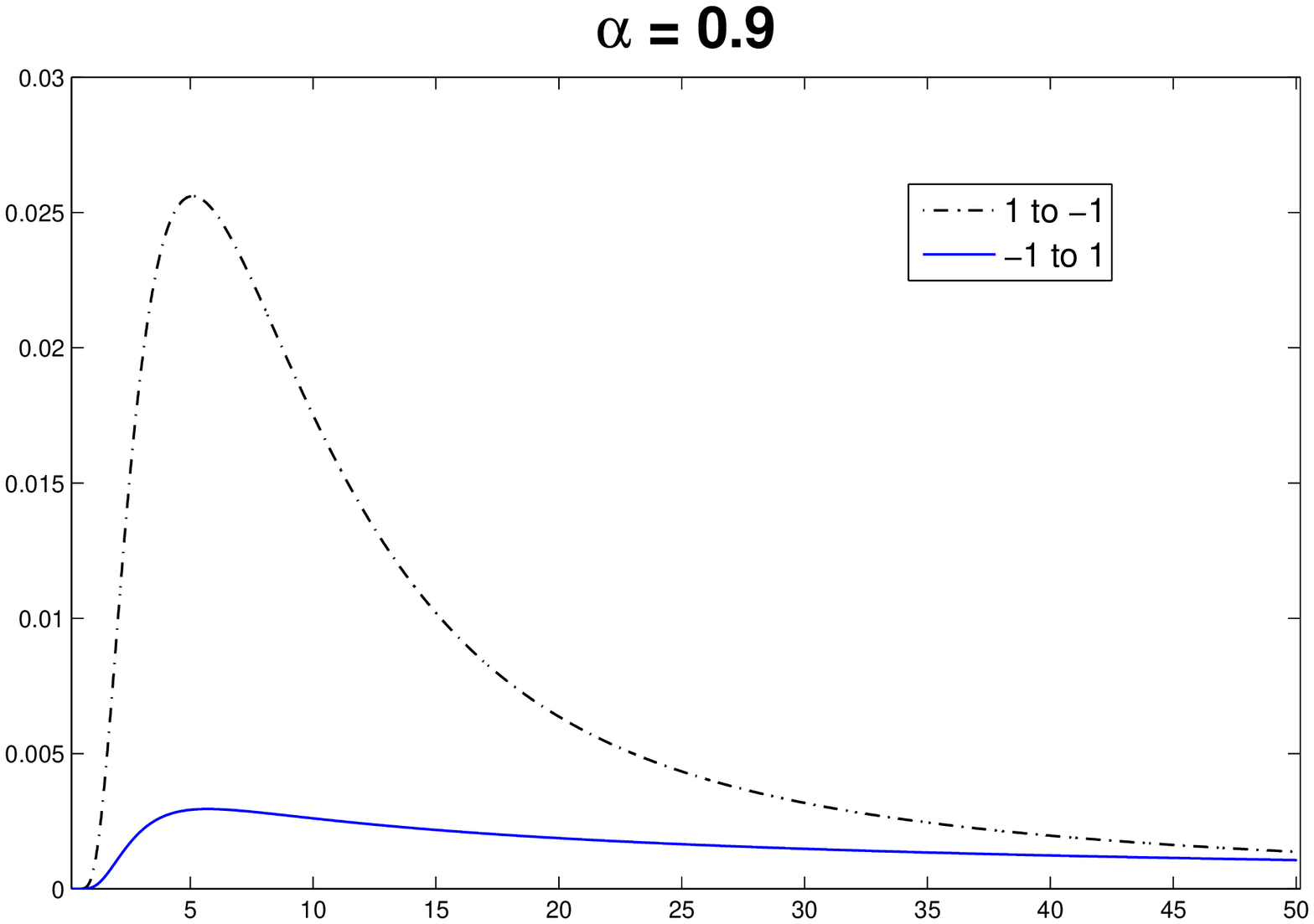} \nudgeup \\
    (c) $\alpha>1/2$ \nudgedown
  \end{tabular}
  \label{ordering}
  \caption{The densities $P_{-1}(T_1^{(\alpha)} \in dt)$ (solid)
    and $P_{1}(T_{-1}^{(\alpha)} \in dt)$ (dashed) for different
    values of $\alpha$.}
\end{figure}

\vspace{1cm}
\acks
The authors wish to express their appreciation to Professor Enrique
Thomann, Professor Edward Waymire and Professor Julia Jones for the
great interest they showed in this work, and for their fruitful
comments and useful advice.


%
%
%
%
\bibliographystyle{apalike}	
\bibliography{myrefs}		

\begin{thebibliography}{}

\bibitem[Appuhamillage et~al., 2009]{app_WRR}
Appuhamillage, T., Bokil, V.~A., Thomann, E., Waymire, E., and Wood, B. (2009).
\newblock {On Skewness in Breakthrough Curves Due to Solute Transport Across an
  Interface}.
\newblock {\em {Water Resour. Res.}}, page http://doi:10.1029/2009WR008258.

\bibitem[Appuhamillage et~al., 2010]{app_AAP}
Appuhamillage, T., Bokil, V.~A., Thomann, E., Waymire, E., and Wood, B. (2010).
\newblock {Occupation and Local Times for Skew Brownian Motion with
  Applications to Dispersion Across an Interface}.
\newblock {\em {Annals of Applied Probability (to appear)}}.

\bibitem[Barlow et~al., 1989]{Barlow89}
Barlow, M., Pitman, J., and Yor, M. (1989).
\newblock On {W}alsh's {B}rownian motions.
\newblock In {\em S\'eminaire de {P}robabilit\'es, {XXIII}}, volume 1372 of
  {\em Lecture Notes in Math.}, pages 275--293. Springer, Berlin.

\bibitem[Berkowitz et~al., 2009]{Berkowitz09}
Berkowitz, B., Cortis, A., Dror, I., and Scher, H. (2009).
\newblock {Laboratory experiments on dispersive transport across interfaces:
  The role of flow direction}.
\newblock {\em {Water Resour. Res.}}, 45(2):http://doi:10.1029/2008WR007342.

\bibitem[Bhattacharya and Waymire, 2009]{bhattway1}
Bhattacharya, R. and Waymire, E. (2009).
\newblock {Stochastic processes with applications}.
\newblock {\em SIAM Classics in Applied Mathematics}.

\bibitem[Blackwell, 1997]{blackwell}
Blackwell, P. (1997).
\newblock {Random diffusion models for animal movement.}
\newblock {\em Ecological Modelling.}, 100:87–102.

\bibitem[Cs\'{a}ki and Hu, 2004]{Csaki}
Cs\'{a}ki, E. and Hu, Y. (2004).
\newblock Invariance principles for ranked excursion lengths and heights.
\newblock {\em Electronic Communications in Probability}, 9:14--21.

\bibitem[Dalziel et~al., 2008]{dalziel}
Dalziel, B.~D., Morales, J.~M., and Fryxell, J.~M. (2008).
\newblock {Fitting probability distributions to animal movement trajectories:
  using artificial neural networks to link distance, resources, and memory.}
\newblock {\em American Naturalist.}, 172:248–258.

\bibitem[Decamps et~al., 2006]{decamps2006}
Decamps, M., Goovaerts, M., and Schoutens, W. (2006).
\newblock {Asymmetric skew Bessel processes and their applications to finance}.
\newblock {\em Journal of Computational and Applied Mathematics},
  186(1):130--147.

\bibitem[Fauchald and Tveraa, 2003]{Fauchald2003}
Fauchald, P. and Tveraa, T. (2003).
\newblock {Using first-passage time in the analysis of area-restricted search
  and habitat selection}.
\newblock {\em Ecology}, 84(2):282--288.

\bibitem[Feller, 1968]{feller}
Feller, W. (1968).
\newblock {An introduction to probability theory and its applications. Vol. I.}
\newblock {\em Oxford, England: Wiley}.

\bibitem[It\^{o} and McKean, 1963]{Ito_McKean63}
It\^{o}, K. and McKean, H. (1963).
\newblock Brownian motions on a half line.
\newblock {\em Ill. J. Math.}, 7(1):181--231.

\bibitem[Le~Gall, 1984]{LeGall}
Le~Gall, J.-F. (1984).
\newblock One-dimensional stochastic differential equations involving the local
  times of the unknown process.
\newblock In {\em Stochastic analysis and applications (Swansea, 1983)}, volume
  1095 of {\em Lecture Notes in Math.}, pages 51--82. Springer, Berlin.

\bibitem[McKenzie et~al., 2009]{McKenzie}
McKenzie, H.~W., Lewis, M.~A., and Merrill, E.~H. (2009).
\newblock {First passage time analysis of animal movement and insights into the
  functional response.}
\newblock {\em Bull. Math. Biol.}, 71(1):107--129.

\bibitem[Ouknine, 1990]{Ouknine}
Ouknine, Y. (1990).
\newblock Le ``{S}kew-{B}rownian motion'' et les processus qui en d\'erivent.
\newblock {\em Teor. Veroyatnost. i Primenen.}, 35(1):173--179.

\bibitem[Pinaud et~al., 2005]{pinaud}
Pinaud, D., Cherel, Y., and Weimerskirch, H. (2005).
\newblock {Effect of environmental variability on habitat selection, diet,
  provisioning behaviour and chick growth in yellow-nosed albatrosses.}
\newblock {\em Marine Ecology Progress Series.}, 298:295--304.

\bibitem[Pitman and Yor, 2001]{Yor}
Pitman, J. and Yor, M. (2001).
\newblock On the distribution of ranked heights of excursions of a brownian
  bridge.
\newblock {\em The Annals of Probability}, 29:361--384.

\bibitem[Polovina et~al., 2001]{polovina}
Polovina, J.~J., Howell, E., Kobayashi, D.~R., and Seki, M.~P. (2001).
\newblock {The transition zone chlorophyll front, a dynamic global feature
  defining migration and forage habitat for marine resources.}
\newblock {\em Progress in Oceanography}, 49:469–483.

\bibitem[Ramirez et~al., 2008]{Ramirez08}
Ramirez, J., Thomann, E., Waymire, E., Chastanet, J., and Wood, B. (2008).
\newblock {A note on the theoretical foundations of particle tracking methods
  in heterogeneous porous media}.
\newblock {\em {Water Resour. Res.}}, 44(1):W01501,doi:10.1029/2007WR005914.

\bibitem[Ramirez et~al., 2006]{Ramirez06}
Ramirez, J., Thomann, E., Waymire, E., Haggerty, R., and Wood, B. (2006).
\newblock {A generalized Taylor-Aris formula and skew diffusion}.
\newblock {\em SIAM J. Multiscale Modeling and Simulation}, 5(3):786--801.

\bibitem[Ramirez, 2010]{Ramirez10}
Ramirez, J.~M. (2010).
\newblock {Multi-skew Brownian motion and diffusion in layered media}.
\newblock {\em (preprint)}.

\bibitem[Schultz and Crone, 2001]{butterfly}
Schultz, C.~B. and Crone, E.~E. (2001).
\newblock {Edge-Mediated Dispersal Behavior in a Prairie Butterfly.}
\newblock {\em Ecology}, 82:1879–1892.

\bibitem[Turchin, 1991]{turchin1991}
Turchin, P. (1991).
\newblock {Translating Foraging Movements in Heterogeneous Environments into
  the Spatial Distribution of Foragers.}
\newblock {\em Ecology.}, 72(4):1253–1266.

\bibitem[Turchin, 1996]{turchin1996}
Turchin, P. (1996).
\newblock {Fractal Analyses of Animal Movement: A Critique.}
\newblock {\em Ecology.}, 77(7):2086--2090.

\bibitem[Walsh, 1978]{walsh1}
Walsh, J.~B. (1978).
\newblock A diffusion with a discontinuous local time.
\newblock {\em Asterisque}, 52--53:37--45.

\bibitem[Wiens and Milne, 1989]{Wiens1989}
Wiens, J.~a. and Milne, B.~T. (1989).
\newblock {Scaling of landscapes in landscape ecology, or, landscape ecology
  from a beetle's perspective}.
\newblock {\em Landscape Ecology}, 3(2):87--96.

\end{thebibliography}






\end{document}